\documentclass[a4paper,12pt]{article}

\usepackage{times}
\usepackage{amsfonts}
\usepackage{amsthm}
\usepackage{amsmath}
\usepackage{epic}
\usepackage{times}

\usepackage{enumerate}

\usepackage[margin=3cm]{geometry}
\usepackage{amsthm,amsmath}
\usepackage{fullpage}

\usepackage{tikz}
\usetikzlibrary{snakes}
\usepackage{soul}

\newcommand{\bw}{\omega}

\newcommand{\abs}[1]{\bigl\lvert #1 \bigr \rvert }
\newcommand{\nats}{{\mathbb N}}
\newcommand{\ints}{{\mathbb Z}}

\newcommand{\abp}{\rho^{\mbox{ab}}}
\newcommand{\p}{\rho}

\newcommand{\card}{{\rm card}}

\def\fpart#1{ \{ #1 \}}

\newtheorem{thm}{Theorem}
\newtheorem{lemm}[thm]{Lemma}
\newtheorem{ques}[thm]{Question}
\newtheorem{theorem}{Theorem}[section]
\newtheorem{lemma}[theorem]{Lemma}
\newtheorem{corollary}[theorem]{Corollary}

\theoremstyle{definition}	 

\newtheorem{remark}[theorem]{Remark}
\newtheorem{problem}{Open problem}

\newcommand{\ignore}[1]{}

\begin{document}

\title{Abelian Complexity in Minimal Subshifts}
\author{Gw{\'e}na{\"e}l Richomme%
\footnote{Universit\'e de Picardie Jules Verne, Laboratoire MIS, 33 Rue Saint Leu, F-80039 Amiens cedex 1, France and
Universit\'e Paul-Val\'ery Montpellier 3, Route de Mende, F-34199 Montpellier cedex 5, France, E-mail: \texttt{gwenael.richomme@univ-montp3.fr}}
\and 
Kalle Saari%
\footnote{Department of Mathematics,
University of Turku,
FI-20014, Finland,
email: \texttt{kasaar@utu.fi}}
\and
Luca Q. Zamboni%
\footnote{
Universit\'e de Lyon, Universit\'e Lyon 1, CNRS UMR 5208 Institut Camille Jordan, B\^atiment du
Doyen Jean Braconnier, 43, blvd du 11 novembre 1918, F-69622 Villeurbanne Cedex, France, email: \texttt{zamboni@math.univ-lyon1.fr}
\indent Reykjavik University, School of Computer Science, Kringlan 1, 103 Reykjavik, Iceland,
email: \texttt{lqz@ru.is}}
}

\maketitle

\abstract{
In this paper we undertake the general study of the Abelian complexity of an infinite word
on a finite alphabet.
We investigate both similarities and differences between the Abelian complexity and the usual subword complexity.  
While the Thue-Morse minimal subshift is neither characterized by its Abelian complexity nor by its subword complexity alone, we show that the subshift is completely characterized by the two complexity functions together.
We 
give an affirmative  answer to an old question of G.~Rauzy by exhibiting a class of words whose Abelian complexity is everywhere equal to $3.$ We also investigate links between  Abelian complexity and the existence of Abelian powers. Using van der Waerden's Theorem, we show that any minimal subshift having bounded Abelian complexity contains Abelian $k$-powers for every positive integer $k$.
In the case of Sturmian words we prove something stronger: For every Sturmian word $\omega$ and positive integer $k,$   each sufficiently long factor of $\omega$  begins in an Abelian k-power.

\section{\label{overview}Introduction}

In this paper we undertake the general study of the Abelian complexity of an infinite word on a finite alphabet. Although some of the topics outlined in this paper have already appeared in the literature (see \cite{CovHed,KaboreTapsoba2008TIA,Rauzy2}), 
to date very little is known on the general Abelian theory of words.
In fact, prior to this paper, some of the key notions  had not even been formally defined.


Given a finite non-empty set $A$ (called the {\it alphabet}), we denote by $A^*$, $A^\nats$ and $A^\ints$ respectively the set of finite words, the set of (right) infinite words, and the set of bi-infinite words over the alphabet $A$. 
We endow $A^\nats$ with the topology generated by the metric
\[d(x, y)=\frac 1{2^n}\,\,\mbox{where} \,\, n=\inf\{k :x_k\neq y_k\}\] 
whenever $x=(x_n)_{n\in \nats}$ and $y=(y_n)_{n\in \nats}$ are two elements of $A^\nats.$ Let $T:A^\nats \rightarrow A^\nats$ denote the {\it shift} transformation defined by $T: (x_n)_{n\in \nats}\mapsto (x_{n+1})_{n\in \nats}.$ By a {\it subshift} on $A$ we mean a pair $(X,T)$ where $X$ is a closed and $T$-invariant subset of $A^\nats.$ A subshift $(X,T)$ is said to be {\it minimal}
whenever $X$ and the empty set are the only $T$-invariant closed subsets of $X.$

Given a finite word $u =a_1a_2\ldots a_n$ with $n \geq 1$ and $a_i \in A,$ we denote the length $n$ of $u$ by $|u|.$ The  \textit{empty word} will be denoted by $\varepsilon$ and we set $|\varepsilon|=0.$   For each $a\in A,$ we let $|u|_a$  denote the number of occurrences of the letter $a$ in $u.$ 
Two words $u$ and $v$ in $A^*$ are said to be {\it Abelian equivalent,} denoted $u\sim_{\mbox{ab}} v,$  if and only if $|u|_a=|v|_a$ for all $a\in A.$ It is readily verified that $\sim_{\mbox{ab}}$ defines an equivalence relation on $A^*.$ 


Given an infinite word
$\omega= \omega _0\omega_1 \omega_2\ldots \in A^\nats$ with $\omega_i\in A$,
we denote by ${\mathcal F}_{\omega}(n)$ the set of all {\it factors} of $\omega$ of length $n,$ that is, the set of all finite words of the form $\omega_{i}\omega_{i+1}\cdots \omega_{i+n-1}$ with $i\geq 0.$
We set \[\p_{\omega}(n)=\mbox{Card}({\mathcal F}_{\omega}(n)).\] The function $\p_{\omega}:\nats \rightarrow \nats$ is called the {\it subword complexity function} of $\omega.$
Given a minimal subshift $(X,T)$ on $A,$ we have that
${\mathcal F}_{\omega}(n)={\mathcal F}_{\omega'}(n)$ for
all $\omega, \omega '\in X$ and $n\in \nats.$ Thus we can define the subword complexity $\p_{(X,T)}(n)$ of a minimal subshift $(X,T)$ by
\[ \p_{(X,T)}(n)=\p_{\omega}(n)\]
for any $\omega \in X.$

Analogously we define ${\mathcal F}^{\mbox{ab}}_{\omega}(n)={\mathcal F}_{\omega}(n)/\sim_{\mbox{ab}}$ and set \[\abp_{\omega}(n)=\mbox{Card}({\mathcal 
F}^{\mbox{ab}}_{\omega}(n)).\]
The function
$\abp_{\omega} :\nats \rightarrow \nats$ which counts the number of pairwise non Abelian equivalent factors of $\omega$ of length $n$ is called the {\it Abelian complexity} of $\omega,$ or {\it ab-complexity} for short\footnote{A different, yet related notion of Abelian complexity, called  {\it Parikh complexity}, was considered in \cite{Ilie2001}.}. As in the case of subword complexity, the definition of Abelian complexity naturally extends to the context of a minimal subshift.

In most instances, the alphabet $A$ will consist of the numbers $\{0,1,2,\ldots ,k-1\}.$  In this case, for each $u\in A^*,$ we denote by
$\Psi(u)$ the {\it Parikh vector} \footnote{Parikh matrices, an extension of Parikh vectors, were recently introduced to characterize words in terms of the occurrences of some scattered subwords (see, \textit{e.g.}, \cite{MaatescuSalomaaSalomaaYu2001TIA,FosseRichomme2004IPL,Salomaa2008TCS}).}  associated to $u,$ that is 
\[\Psi(u)=(|u|_0,|u|_1,|u|_2,\ldots ,|u|_{k-1}).\]
Given an infinite word $\omega \in A^\nats$ we set
\[\Psi_{\omega}(n)=\{\Psi(u)\,|\, u \in {\mathcal F}_{\omega}(n)\}\] and thus we have  $\abp_{\omega}(n)=\mbox{Card}(\Psi_{\omega}(n)).$

There are a number of similarities  between the usual subword complexity of an infinite word and its Abelian complexity. 
For instance both may be used to characterize periodic bi-infinite words. 
Here a word $\omega$  is \emph{periodic} if there exists a positive integer $p$ such that 
$\omega_{i+p} = \omega_i$ for all indices $i$, and it is \emph{ultimately periodic} if $\omega_{i+p} = \omega_i$ for all sufficiently large $i$.
An infinite word is \emph{aperiodic} if it is not ultimately periodic.

\begin{thm}[]
Let $\omega$ be a bi-infinite word over the alphabet $A$. The following properties hold.
\begin{itemize}
\item (M. Morse, G.A. Hedlund, \cite{MorHed1940})
The word $\omega$ is periodic if and only if  $\p_{\omega}(n_0)\leq n_0$ for some $n_0\geq 1$. 
\item (E.M. Coven and G.A. Hedlund, \cite{CovHed} )
The word $\omega$ is periodic  if and only if  $\abp_{\omega}(n_0)=1$ for some $n_0\geq 1$. (See also Lemma~\ref{L:aperiodic})
\end{itemize}
\end{thm}




 
\noindent The condition in the first item of the previous theorem also gives rise to a characterization of ultimately periodic words by means of subword complexity.
Abelian complexity does not, however, yield such a characterization. Indeed, both Sturmian words (see below) and
the ultimately periodic word  $01^\infty = 0111\cdots$  have the same, constant $2$, Abelian complexity.

As another example, both complexity functions may be used to characterize \textit{Sturmian} words:
%

\begin{thm}\label{SturmianChar}
Let $\omega$ be an aperiodic infinite word over the alphabet $\{0,1\}$. The following conditions are equivalent:
\begin{itemize}
\item The word $\omega$ is balanced, that is, {\it Sturmian}.
\item (M. Morse, G.A. Hedlund, \cite{MorHed1940}). The word $\omega$ satisfies $\p_{\omega}(n_0) = n+1$ for all $n \geq 0$. 
\item (E.M. Coven, G.A. Hedlund, \cite{CovHed}). The word $\omega$ satisfies $\abp_{\omega}(n)=2$ for all $n\geq 1$. (See also the discussion in the beginning  of Section~\ref{RauzyQuestion}).
\end{itemize}
\end{thm}

We recall an infinite word $\omega \in A^\nats$ is said to be $C$-{\it balanced}  ($C$ a positive integer) if $||U|_a-|V|_a|\leq C$ for all $a\in A$ and all factors $U$ and $V$ of $\omega$ of equal length. A word $\omega$
is called {\it balanced} it it is $1$-balanced.

%

%

However, in some cases the two complexities  exhibit radically different behaviors. For instance, as was originally pointed out to us by P. Arnoux \cite{Arnoux2008}, consider the binary Champernowne word 
\[ {\mathcal C} =01101110010111011110001001\ldots\]
obtained by concatenating the binary representation 
of the consecutive natural numbers. Let $\omega$ denote the morphic image of $\mathcal C$
 under the Thue-Morse morphism $\mu$ defined by $0\mapsto 01$ and $1\mapsto 10$.  Then while $\p_{\omega}(n)$ has exponential growth, we will see in Theorem~\ref{T:caracAbCompTM} that $\abp_{\omega}(n)\leq 3$ for all $n.$

In contrast, there exist infinite words having linear subword complexity and unbounded ab-complexity. 
In fact, in \cite{CassFerZam}, the third author together with J. Cassaigne and S. Ferenczi  established the existence of a word with subword complexity $\p(n)=2n+1$ which is not $C$-balanced for any positive integer $C.$ 
An infinite word $\omega \in A^\nats$ is said to be $C$-{\it balanced}  ($C$ a positive integer) if $||U|_a-|V|_a|\leq C$ for all $a\in A$ and all factors $U$ and $V$ of $\omega$ of equal length. We prove that

\begin{lemm}
$\abp_{\omega}(n)$ is bounded if and only if $\omega$ is $C$-balanced for some positive integer $C.$
\end{lemm}

\noindent Hence the word constructed in \cite{CassFerZam} of subword complexity $\p(n)=2n+1$  has  unbounded
 ab-complexity.
 
Finally, in some cases, the two complexity functions work in tandem to give explicit characterizations of certain minimal subshifts. For instance we prove that

\begin{thm} Let $\mathbf{TM}_0\in \{0,1\}^\nats$ denote  the fixed point  beginning in $0$ of the Thue-Morse morphism $0\mapsto 01,$ $1\mapsto 10,$ and let $\omega$ be a recurrent infinite word.
Then
\[\p_{\omega}(n)=\p_{\mathbf{TM}_0}(n)\,\,\, n\in \nats\]
and
\[\abp_{\omega}(n)=\abp_{\mathbf{TM}_0}(n)\,\,\, n\in \nats\]
if and only if  $\omega$ is in the subshift generated by $\mathbf{TM}_0.$

\end{thm}

\noindent Recall that an infinite word is \textit{recurrent} if each of its factors occurs infinitely often in $w$. 

\vspace{.1in}

Inspired by the last characterization of Sturmian words given in Theorem~\ref{SturmianChar},  G.~Rauzy asked the following question:

\begin{ques}[G. Rauzy, \cite{Rauzy2}] Does there exist an infinite word $\omega$
whose ab-complexity $\abp_{\omega}(n)=3$ for all $n\geq 1.$ 
\end{ques}

He suggests that the likely answer to this question is NO. However, we give two positive solutions to Rauzy's question:

\begin{thm} Let $\omega $ be an aperiodic balanced word on the alphabet $\{0,1,2\}.$ Then $\abp_{\omega}(n)=3$ for all $n\geq 1.$
\end{thm}

\begin{thm} Let $\omega'\in \{0,1\}^\nats$ be any aperiodic infinite word, and let $\omega$ be the image of $\omega'$ under the morphism
$0\mapsto  012,$ and $ 1\mapsto  021.$ Then $\abp_{\omega}(n)=3$ for all $n\geq 1.$
\end{thm}

In Section~\ref{sec:repet}, we investigate a surprising connection between Abelian complexity and avoidance of Abelian powers. By an \textit{Abelian} $k${\it-power} 
we mean a non-empty word of the  form $w=u_1u_2\cdots u_k$ where the words $u_i$ are pairwise Abelian equivalent.
In this case we  say $w$ has an \emph{Abelian period} equal to $\abs{u_1}$.
Powers, or repetitions, occurring in an infinite word is a topic of great interest having applications to a broad range of areas (see, \textit{e.g.}, \cite{AB, AlDaQuZa, AlZa,Lothaire2005}). One stream of research dating back to the works of Thue \cite{Thue1906,Thue1912} is the study of patterns avoidable by infinite words (see, \textit{e.g.}, \cite{Lothaire1983book,Lothaire2002,RichommeSeebold2004IJFCS,RichommeWlazinski2007DAM,BerstelLauveReutenauerSaliola2008book}). 
In the Abelian context, F.M.~Dekking \cite{Dekking1979} showed that Abelian 4-powers are avoidable on a binary  alphabet and that Abelian cubes are avoidable on a $3$-letter alphabet. V.~Ker\"anen \cite{Keranen1992ICALP} proved that Abelian squares are avoidable on four letters. 
We prove the following general result relating Abelian complexity and Abelian repetitions.

\begin{thm}
Let $\omega$ be an infinite word on a finite alphabet 
having bounded ab-complexity. 
Then $\omega$ contains an Abelian $k$-power for every positive integer $k.$
\end{thm}

In view of Theorem~\ref{SturmianChar}, the previous theorem implies that Sturmian words admit occurrences of Abelian $k$-powers for all $k$.
In Section~\ref{SturmianSec} we show that actually these words possess a  much stronger property:
\begin{thm}
For every Sturmian word $\omega$ and every positive integer $k$, there exist two integers $\ell_1$ and $\ell_2$ such that 
each position in $\omega$ has an occurrence of an Abelian $k$-power  with Abelian period $\ell_1$ or $\ell_2$.
In particular, every Sturmian word begins in an Abelian k-power for all positive integers $k$.
\end{thm}

\paragraph{Acknowledgements} 
The second author is partially supported by grant no. 8121419 from the Finnish Academy. The third author is partially supported by grant no. 090038011 from the Icelandic Research Fund.

\section{Generalities}\label{sec:generalities}

As explained in Section~\ref{overview}, the Abelian complexity $\abp_\bw$ of a word $\bw$ is the function which counts the number of pairwise non Abelian equivalent factors of $\omega$ for each length $n$. In other words, for all $n \geq 0$, $\abp_\bw(n)$ is the cardinality of the set of Parikh vectors of factors of length $n$ of $\bw$.

Let $a, b$ be two letters in $A = \{0, \ldots, p-1\}$ and let $u$ be a word over $A$. If $a = b$, $\Psi(au) = \Psi(ub)$. When $a \neq b$, $\Psi(au) - \Psi(ub)$ is the vector whose all entries are 0 except its $(a+1)$-th entry with value $+1$ and its $(b+1)$-th entry with value $-1$. This shows how Parikh vectors evolve when considering two successive factors of same length of a word $\bw$. As an immediate consequence, we deduce the following fact that will be used very often implicitly.

\begin{lemma}\label{F:basic}
If an infinite word $\bw$ has two factors $u$ and $v$ of same length $n$ for which the $i$th entry of the Parikh vector are p and p+c respectively for some $p$ and $c$ then for all $\ell = 0, \ldots, c$, there exist factors $u_\ell$ of $\bw$ whose $i$th entry is $p+\ell$.
\end{lemma}

With the notation of this fact, we can see that $\abp_\omega(n) \geq c+1$. This implies:

\begin{lemma}\label{bounded=balanced}
For a word $\bw \in A^\nats \cup A^\ints$, the function $\abp_\bw$ 
is bounded if and only if $\bw$ is $C$-balanced for some positive integer $C$.
\end{lemma}

\begin{proof}
If $\abp_\bw$ is bounded by $K$, then $\bw$ is $(K-1)$-balanced. If $\bw$ is $C$-balanced, then entries of Parikh vectors of factors of $\bw$ can take at most $C+1$ values, so that $\abp_\bw(n) \leq (C+1)^{\card(A)}$.
\end{proof}


A natural question concerns the possible extremal values of the Abelian complexity. As shown by next result, minimal values are reached by periodic words.

\begin{lemma}[E.M. Coven and G.A. Hedlund, {\cite[Remark 4.07]{CovHed}}]
\label{L:aperiodic}
Let $\omega\in A^\nats \cup A^{\ints}$
be a right infinite or a bi-infinite word. Then  
$\omega$ is periodic  of period $p$ if and only if 
$\abp_{\omega}(p)=1$. 
\end{lemma}


Let us recall that a word $u$ is a \textit{right special factor} of an infinite word $\omega$ if there exist distinct letters $\alpha$ and $\beta$ such that $u\alpha$ and $u\beta$ are both factors of $\omega$. 

\begin{proof}[Proof of Lemma~\ref{L:aperiodic}]
The ``only if'' part  is immediate. The other direction can be deduced from the observation that a non-periodic word $\omega$ must contain arbitrarily long right special factors implying $\abp_\omega(n) \geq 2$ for all $n \geq 1$.
\end{proof}

\medskip

Concerning the maximal Abelian complexity, it is clear that it is reached by any infinite word containing all finite words as factors, as for instance the Champernowne word. Let us denote $\abp_{\rm max}$ the Abelian complexity of such a word. Since, for any word $u$ of length $n$ over a $k$-letter alphabet, $\Psi(u)$ is a $k$-tuple $(i_1, \ldots, i_k)$ with $n = i_1 + i_2 + \ldots + i_k$, $\abp_{\rm max}$ is the maximum number of ways of writing
 $n$ as the sum of $k$ nonnegative integers. This well-known number (see, \textit{e.g.}, \cite{WeissteinWeb}) is called the number of compositions of $n$ into $k$ parts and and its value is given by the binomial coefficient 
 $\binom{n+k-1}{k-1}$. 
 This can be summarized:
\begin{theorem}
For all infinite words $\bw$ over a $k$-letter alphabet, and for all $n \geq 0$,
\[
\abp_\bw(n) \leq \abp_{\rm max}(n) = \binom{n+k-1}{k-1}.
\]
In particular, the ab-complexity is in $O(n^{k-1})$. 
\end{theorem}




In Section~\ref{overview} we provided an example of an infinite word having an exponential subword complexity but a linear Abelian complexity. Here follows an example of a binary infinite word having maximal Abelian complexity but a linear subword complexity. Let $f$ and $g$ be the morphisms defined by $f(a) = abc$, $f(b) = bbb$, $f(c) = ccc$, $g(a) = 0 = g(c)$ and $g(b) = 1$. The image $\bw$ by $g$ of the fixed point of $f$ on $a$ is the word $ 0\prod_{i \geq 0} 1^{3^i}0^{3^i}$. That $\abp_\bw = \abp_{\rm max}$ is immediate. Since $w$ is an automatic sequence, it has  linear subword complexity (see Theorems 6.3.2 and 10.3.1 in \cite{AlloucheShallit2003book}).

\section{\label{subsec:abcompTM}Abelian Complexity of the Thue-Morse Word}

Let us recall that the celebrated Thue-Morse word is the fixed point $\mathbf{TM}_0$ beginning in $0$ of the Thue-Morse morphism $\mu$ defined by $\mu(0) = 01$ and $\mu(1) = 10.$  Next theorem gives its Abelian complexity:

\begin{theorem} \label{T:Thue-Morse ab-complexity} 
$\abp_{\mathbf{TM}_0}(n)=2$ for $n$ odd and $\abp_{\mathbf{TM}_0}(n)=3$ for $n \neq 0$ even.
\end{theorem}

This result is a direct corollary of Theorem~\ref{T:caracAbCompTM} below which shows that this complexity is only due to the fact that $\mathbf{TM}_0$ is the image of a word by $\mu$. Moreover Theorem~\ref{T:caracAbCompTM} characterizes the class of all words having same Abelian complexity as the Thue-Morse word. Corollary~\ref{C:More on Thue-Morse} shows that the subshift generated by the Thue-Morse word can be distinguished in this class by its subword complexity.

We need a preliminary result:

\begin{lemma}\label{aper} Let $\bw =\bw_0\bw_1\bw_2\ldots\in \{0,1\}^\nats$ be an aperiodic infinite word. Then for every $k\geq 2,$ there exist factors $U$ and $V$ of $\bw$ of length $k,$ with $U=0u1$ and $V=1v0$ for some words $u$ and $v$ (possibly empty).
\end{lemma}

\begin{proof} 
Suppose for some $k\geq 2$, the aperiodic word $\bw$ does not contain a factor of length $k$ beginning in $0$ and ending in $1.$
This implies that for each $i\geq 0,$ if $\bw_i=0,$ then  $\bw_{i+k-1}=0.$ 
Thus for $i$ sufficiently large we have that $\bw_{i+k-1}=\bw_i,$ and hence $\bw$
is eventually periodic, a contradiction. Similarly, if $\bw$ does not contain a factor $V$ of length $k$ beginning in $1$ and ending in $0,$ it would follow that $\bw$ is eventually periodic.

\end{proof}

\begin{theorem}\label{T:caracAbCompTM} 
The Abelian complexity of an aperiodic word $\bw \in \{0,1\}^\nats$ is 
$$\left\{
\begin{tabular}{l}
$\abp_\bw(n) = 2$ for $n$ odd,\\
$\abp_\bw(n) = 3$ for $n \neq 0$ even,\\
\end{tabular}\right.$$
if and only if there exists a word $\bw'$ such that $\bw = \mu(\bw')$, $\bw = 0\mu(\bw')$ or $\bw = 1\mu(\bw')$.
\end{theorem}

\begin{proof} 
Assume first that the Abelian complexity of a word $\bw$ is $2$ for $n$ odd and $3$ for $n \neq 0$ even.
We first prove by induction that  for all $k \geq 0$
$$\left\{
\begin{tabular}{l}
$\Psi_\bw(2k+1) = \{ (k+1, k), (k, k+1)\}$ and\\
  $\Psi_\bw(2k+2) = \{ (k+1, k+1), (k, k+2), (k+2, k)\}$.
\end{tabular}\right.$$
The result is true for $k = 0$. 
Assume that $\Psi_\bw(2k+2)$ is as above. Since $\abp_\bw(2k+3) = 2$, by Lemma~\ref{F:basic}
there exist integers $p$ and $q$ such that 
$$\Psi_\bw(2k+3) = \{(p+1, q), (p, q+1)\}.$$ 
Moreover 
\begin{eqnarray*}
\Psi_\bw(2k+3) & \subseteq & \Psi_\bw(2k+2) + \{ (1, 0), (0,1) \}\\
&= & \{ (k+2, k+1), (k+1, k+2), (k, k+3), (k+3,k)\}.
\end{eqnarray*} 
Thus exactly one of the following three identities hold:
\begin{align*}
\Psi_\bw(2k+3) &= \{(k+2, k+1), (k+1, k+2)\}, \\
\Psi_\bw(2k+3) &= \{(k+2, k+1), (k+3, k)\},   \\
\Psi_\bw(2k+3) &= \{(k+1, k+2), (k, k+3)\}.
\end{align*}
The second identity is not possible, though, because $(k,k+2) \in \Psi_\bw(2k+2)$. Similarly we see that the third identity cannot hold, and therefore
the first one does.
Now 
\begin{eqnarray*}
\Psi_\bw(2k+4) &\subseteq & \Psi_\bw(2k+3) + \{ (1, 0), (0,1) \}\\ 
&=& \{ (k+3, k+1), (k+2, k+2), (k+1, k+3)\}.
\end{eqnarray*} 
Since $\abp_\bw(2k+4) = 3$, the previous inclusion is an equality, and the inductive part is proved.

It follows from $\Psi_\bw(3) = \{ (1,2), (2,1)\}$ that 000 and 111 do not occur in $w$.
Assume that $\bw$ has one factor of the form $11(01)^k011$. This word has Parikh vector $(k+4, k+1)$ which is not possible from what precedes. Hence $\bw$ has no factor of the form $11(01)^k011$ and similarly no factor of the form $00(10)^k100$ (that is between two consecutive occurrences of $11$ there is an occurrence of $00$ and between two occurrences of $00$ there is an occurrence of $11$). Thus, since $\bw$ is aperiodic, both $00$ and $11$ must occur infinitely often, so that $\bw$ can be decomposed $pu_1 u_2 \ldots u_k \ldots$ with each $u_i$ in $1(01)^*00(10)^*1$ and $p$ a suffix of a word in $1(01)^*00(10)^*1$. This shows that $\bw = \mu(\bw')$, $\bw = 0\mu(\bw')$ or $\bw = 1\mu(\bw')$ for a word $\bw'$.


\medskip

Assume from now on that $\bw = \mu(\bw')$, $\bw = 0\mu(\bw')$ or $\bw = 1\mu(\bw')$ for a word $\bw'$.

Let $n=2k+1.$ Then every factor $U$ of $\bw$ of length $n$ is either of the form $\mu (u)a$ or $a\mu(u)$ for some factor $u$ of $\bw'$ of length $k$ and some $a\in \{0,1\}.$  It follows that \[\Psi_{\bw}(n)\subseteq \{(k+1,k), (k, k+1)\}.\]  Since $\bw$ is aperiodic, by Lemma~\ref{L:aperiodic} we have that  $\abp_{\bw}(n)\geq 2$  and hence
$\abp_{\bw}(n)=2.$

Next suppose $n=2k.$ Then every factor $U$ of length $n$ is either of the form
$\mu(u)$ or $a\mu(v)b$ for some  factors $u$ and $v$ of $\bw'$ of length $k$ and $k-1$ respectively, and for some $a,b\in \{0,1\}.$ It follows that
\[\Psi_{\bw}(n)\subseteq \{(k,k), (k-1,k+1),(k+1,k-1)\}.\]
It follows from Lemma~\ref{aper} that $\bw'$ contains factors
$u'$ and $v'$ of length $k-1$ with $u'$ preceded by $0$ and followed by $1$
and $v'$ preceded by $1$ and followed by $0.$ 
Thus $\bw$ contains the factors of length $n$ of the form $1\mu(u')1$ and $0\mu(v')0,$ whose
corresponding Parikh vectors are $ (k-1,k+1),$ and $(k+1,k-1)$ respectively.
Hence 
\[\Psi_{\bw}(n)= \{(k,k), (k-1,k+1),(k+1,k-1)\}\]
whence $\abp_{\bw}(n)=3.$

\end{proof}

\begin{remark}
The proof of the ``only if'' part of Theorem~\ref{T:caracAbCompTM} easily extends to ultimately periodic words as follows. 
If an ultimately periodic word $\omega$ has the same Abelian complexity as the Thue--Morse word, then 
 $\omega$ can be decomposed into $pu_1\ldots u_k (01)^\infty$ or $pu_1\ldots u_k (10)^\infty$ with words $p$ and $u_i$ as in the proof of Theorem~\ref{T:caracAbCompTM}. (By the notation $u^\infty$ we mean the infinite word $uuu \cdots$)
That is, $\omega$ is of the form $\mu(v b^{\infty})$ or 
$a\mu(v b^{\infty})$ with $a,b \in \{0,1\}$ and for some finite word $v$.

 The converse part of Theorem~\ref{T:caracAbCompTM}, however,  does not extend to ultimately periodic words. Indeed,  the word $(01)^\infty = \mu(0^\infty)$ is of the correct form, but it does not have the same Abelian complexity as the Thue-Morse word.
\end{remark}

\noindent We  end this section with:

\begin{corollary}\label{C:More on Thue-Morse} Let $\mathbf{TM}_0\in \{0,1\}^\nats$ denote  the Thue-Morse word beginning in $0,$ and let $\omega$ be a recurrent infinite word.
Then
\[\p_{\omega}(n)=\p_{\mathbf{TM}_0}(n)\,\,\, n\in \nats\]
and
\[\abp_{\omega}(n)=\abp_{\mathbf{TM}_0}(n)\,\,\, n\in \nats\]
if and only if  $\omega$ is in the subshift generated by $\mathbf{TM}_0.$

\end{corollary}

\begin{proof} The ``only if'' part is clear. For the converse, we recall that in  
\cite{AbeBrl} it is shown that every recurrent infinite word  $\omega \in \{0,1\}^\nats$ whose subword complexity is equal to that of $\mathbf{TM}_0$ is either in the shift orbit closure of $\mathbf{TM}_0$ or is in the shift orbit closure of $\Delta (\mathbf{TM}_0)$ where $\Delta$ is the letter doubling morphism defined by $0\mapsto 00$ and $1\mapsto 11.$ However, if $\omega \in \{0,1\}^\nats$
is in the shift orbit closure of $\Delta (\mathbf{TM}_0),$ then it would contain the $4$ factors $111,110,100,000$ of length $3,$ and hence $\abp_{\omega}(3)=4.$ Thus $\omega$ is in the shift orbit closure of $\mathbf{TM}_0.$

\end{proof}

\section{Two Answers to a Question of G. Rauzy}\label{RauzyQuestion}

Recall that a Sturmian word is an aperiodic binary balanced word. Among the few existing families of words characterized using Abelian complexity, the following one concerning Sturmian words is the earliest one we know.\footnote{I.~Kabor\'e and T. Tapsoba also characterized using Abelian complexity the family of so-called quasi-Sturmian by insertion words (a subclass of the class of infinite words over a three-letter alphabet having subword complexity $n+2$) \cite{KaboreTapsoba2008TIA}. These words defined over a ternary alphabet verify $\abp(n) = 2$ for $n\neq 0$ even and $\abp(n) = 4$ for $n \neq 1$ odd.
} It is essentially due to 
E.M. Coven and  G.A. Hedlund  (an easy consequence of Lemma 4.02 and remark 4.07 in~\cite{CovHed}).

\begin{theorem}[E.M. Coven, G.A. Hedlund]\label{T:Sturm}
Let $\omega$ be a binary right infinite word. Then $\omega$ is aperiodic and balanced (i.e., a Sturmian word) if and only if $\abp_\omega(n)=2$ for all $n\geq 1.$
\end{theorem}

Inspired by this characterization of Sturmian words, Rauzy asked whether there exist aperiodic words on a $3$-letter alphabet such that $\abp(n)=3$ for all $n\geq 1$  (see section 6.2 in \cite{Rauzy2}). Our next two results provide two different solutions to Rauzy's question.

Let $p \geq 3$ be any integer, let $\omega'$ be any Sturmian word over $\{0,1\}$ and let $\omega = (p-1)(p-2)\ldots 2\omega'$ ($\omega$ is written over the alphabet $\{0, 1, \ldots, (p-1)\}$. As a consequence of Theorem~\ref{T:Sturm}, we can see that $\abp_\omega(n) = p$ for all $n \geq 1$ (in particular when $p = 3$). This provides one answer to Rauzy's question. Nevertheless it is not completely satisfactory since $w$ is not recurrent.  The next two results give two distinct classes of \textit{uniformly recurrent words} (each factor occurs infinitely often with bounded gaps) having constant ab-complexity equal to $3$.

\begin{theorem}\label{T:answer Rauzy 1} Let $\omega $ be an aperiodic uniformly recurrent balanced word on the alphabet $\{0,1,2\}.$ Then $\abp_{\omega}(n)=3$ for all $n\geq 1.$

\end{theorem}

\begin{proof} Let $\omega$ be an aperiodic uniformly recurrent balanced word on the alphabet $\{0,1,2\}.$
P.~Hubert showed that up to word isomorphism, $\omega$ 
is obtained from a Sturmian word  ${\mathbf x}\in \{x,y\}^\nats$  by replacing the subsequence of $x$'s in $\mathbf x$ by the periodic word $(01)^\infty$ and the subsequence of $y$'s in $\mathbf x$ by the periodic word $2^\infty $ (see Theorem~1 in \cite{Hu}).  In particular we have that

\begin{enumerate}
\item Deleting all occurrences of $2$ in $\omega$ gives the periodic sequence $(01)^\infty.$ 
\item  $h(\omega)=\mathbf x$ where  $h:\{0,1,2\}\rightarrow \{x,y\}$ is the morphism  defined
by $h(0) = h(1) = x$, and $h(2)=y.$ 
\item If $a\in \{0,1\}$ and $ua$ and $va$ are  distinct factors of $\omega,$  then
$h(u)\neq h(v).$
\end{enumerate}

\noindent Item 3 follows from item 1 together with the fact that if $h(u)=h(v),$ then either $u=v,$ or $|u|=|v|$ and every occurrence of $0$ (respectively $1$) in $u$ is an occurrence of $1$ (respectively $0$) in $v$ (that is $u = (01)^k0$, $v = (10)^k1$ for some $k \geq 0$).

Since $\omega$ is a balanced word, we observe that if it has Parikh vectors $(i, j, k)$ and\\ $(i+1, j-1, k)$ for some integers $i, j, k$ then the only other possible vectors in $\Psi_{\omega}(n)$ are $(i+1,j, k-1)$ and $(i, j-1, k+1)$, and these two vectors cannot occur simultaneously in $\Psi_{\omega}(n)$. 
Consequently $\abp_{\omega}(n)\leq 3$ for all $n\geq 1$. 
In what follows, we will show that $\abp_{\omega}(n)\geq 3$ for all $n\geq 1.$

Applying Theorem 2 in \cite{Hu}, we deduce that $\p_{\omega}(n)=2(n+1)$ for all $n$ sufficiently large. This in turn implies that  for every $n\geq 1,$ either {\bf Case 1:} $\omega$ contains a (right special) factor
$u$ of length $n$ such that each of $u0,u1,u2$ are each factors of $\omega,$ or {\bf Case 2:} $\omega$ contains two (right special) factors $u$ and $v$ of length $n$ such that $ua,ub,vc,vd$ are each factors of $\omega$ for some $a,b,c,d \in \{0,1,2\}$ with $a\neq b$ and $c\neq d.$ In {\bf Case 1}, by considering the Parikh vector associated to the suffix of length $n$ of each of $u0,u1,u2,$ we obtain three distinct Parikh vectors, and hence $\abp_{\omega}(n)\geq 3.$

We next consider {\bf Case 2}. In this case we will show that 
each of $0,1,2$ is in $\{a,b,c,d\}.$ Assume to the contrary that
$2\notin \{a,b,c,d\}.$ Then $\{a,b,c,d\}=\{0,1\}.$  But then by item~1 above, neither $0$ nor $1$ occur in $u$ and in $v,$ and thus $u=2^n$ and $v=2^n$ contradicting that $u\neq v.$ Thus $2\in \{a,b,c,d\}.$
Next suppose that $0\notin \{a,b,c,d\}.$ Then $u1,u2,v1,v2$ are all factors of $\mathbf y.$ It follows by item 3 above that $h(u)\neq h(v).$ But as $h(1)\neq h(2),$ it follows that $h(u)$ and $h(v)$ are distinct right special factors of $\mathbf x$ of length $n,$ contradicting that a Sturmian word has exactly one right special factor of every length. Hence, we have $0\in \{a,b,c,d\}.$ Similarly we deduce that $1\in \{a,b,c,d\}.$ 

Having established that $\{a,b,c,d\}=\{0,1,2\},$  by considering the Parikh vector of $u,v,$ and of the suffix of length $n$ of each of $ ua,ub,vc,vd,$ we obtain at least
three distinct Parikh vectors, and hence $\abp_{\omega}(n)\geq 3.$

 \end{proof}

\begin{theorem}\label{T:answer Rauzy 2} Let $\omega'\in \{0,1\}^\nats$ be any aperiodic infinite word, and let $\omega$ be the image of $\omega'$ under the morphism $f$ defined by
$0\mapsto  012,$ and $ 1\mapsto  021.$ Then $\abp_{\omega}(n)=3$ for all $n\geq 1.$
\end{theorem}

\begin{proof} We first note that:
\begin{enumerate}

\item For every factor $U$ of $\omega,$ there exist a  suffix $x$ (possibly empty) of $012$ or of $021,$ a prefix $y$ (possibly empty) of 
either $012$ or of $021,$ and a factor  $u$ (possibly empty) of $\omega ',$  such that  $U=xf(u)y.$ 

\item By Lemma~\ref{aper}, for every suffix $x$ of $012$ (respectively $021)$ and for every prefix $y$ of $021$ (respectively $012),$ and for every $k\geq 0,$ there exists a factor $u$ of $\omega'$ of length $k$ such that $xf(u)y$ is a factor of $\omega$ of length $3k+|x|+|y|.$

\end{enumerate}

In order to prove that $\abp_{\omega}(n)=3$ we consider separately the case where $n\equiv 0 \pmod 3,$  $n\equiv 1 \pmod 3,$ and $n\equiv 2 \pmod 3. $  In each case, we simply count the number of distinct Parikh vectors associated to words of the form $xf(u)y$ with $x$ and $y$  as in item 1 above and with
$|x|+|y|\equiv n \pmod 3.$ In each case this gives that $\abp_{\omega}(n)\leq 3.$ Similarly, by counting the number of distinct Parikh vectors associated to words of the form $xf(u)y$ with $x$ and $y$ as in item 2 above and with  $|x|+|y|\equiv n \pmod 3,$ we find that  $\abp_{\omega}(n)\geq 3.$

\end{proof}

\noindent Having answered Rauzy's original question, we may now ask:

\begin{problem} Does there exist a recurrent infinite word $\omega$ with
$\abp_{\omega}(n)=4$ for every $n\geq 1$?
\end{problem}

We suspect that the answer is NO! Using Hubert's main characterization of uniformly recurrent aperiodic balanced words given in \cite{Hu},  it can be shown that if 
$\abp_{\omega}(n)=4$ for every $n\geq 1,$ then $\omega$ is not balanced.
While we suspect that the answer to the above question is no,  Aleksi Saarela~\cite{aleksi} recently showed that for every positive integer $k$,
there exists an infinite word whose Abelian complexity at $n$ equals $k$ for all sufficiently large~$n$.

%

\section{\label{sec:repet}Bounded Abelian Complexity}

By a well-known result of V.~Ker\"anen \cite{Keranen1992ICALP}, Abelian squares are avoidable over quaternary alphabets.
This  is to be contrasted with the following result, which says that avoiding Abelian powers while keeping the Abelian complexity bounded is impossible.


\begin{theorem}\label{bounded} 
Let $\omega$ be an infinite word on a finite alphabet 
having bounded ab-complexity. 
Then $\omega$ contains an Abelian $k$-power for every positive integer $k.$
\end{theorem}


\begin{remark}
Unlike with Abelian complexity, low subword complexity does not guarantee existence of repetitions.
For example, the Fibonacci word does not contain any 4th powers~\cite{MigPir1992}, yet its subword complexity is minimal amongst all aperiodic infinite words.
\end{remark}

\noindent We begin proving Theorem~\ref{bounded} with the following lemma:

\begin{lemma}\label{congo} Let $M$ and $r$ be positive integers. Then there exist positive integers $\alpha _1,\alpha_2,\ldots ,\alpha_r$ and $N$ such that whenever
\[\sum _{i=1}^rc_i\alpha_i \equiv 0\pmod {N}\] 
for integers $c_i$ with  $|c_i|\leq M$ for
$1\leq i\leq r,$ then $c_1=c_2=\cdots =c_r=0.$
\end{lemma}

\begin{proof} We define $\alpha_i$ inductively as follows: Set $\alpha _1=1,$
and for $i\geq 1,$ choose \[\alpha_{i+1}>M\sum_{j=1}^{i}\alpha_j.\]
Let $N$ be any integer with $N>M\sum_{j=1}^r \alpha_j.$
Now suppose that
\[\sum _{i=1}^rc_i\alpha_i \equiv 0\pmod {N}\] with each $|c_i|\leq M.$
Then
\[| \sum _{i=1}^rc_i\alpha_i|\leq \sum_{i=1}^r|c_i|\alpha_i\leq M\sum_{i=1}^r\alpha_i<N.\]
Hence
\[\sum _{i=1}^rc_i\alpha_i=0.\]
To see that each $c_i=0,$ suppose to the contrary that some $c_i\neq 0,$ and let $t$ denote the largest positive integer $1\leq t\leq r$ for which $c_t\neq 0.$
If $t=1,$ then we have $c_1=c_1\alpha_1=0,$ a contradiction. Otherwise,  if $t>1,$
\[|c_t\alpha_t|=|-\sum_{i=1}^{t-1}c_i\alpha_i|\leq \sum_{i=1}^{t-1}|c_i|\alpha_i\leq M\sum_{i=1}^{t-1}\alpha_i<\alpha _t\leq |c_t\alpha_t|,\]
a contradiction.

\end{proof}

\noindent We now recall the following well-known result due to van de Waerden:

\begin{theorem}[van der Waerden's theorem - see, \textit{e.g.}, chapter~3 in \cite{Lothaire1983book}]
If $\nats$ is partitioned into $k$ classes, then one of the classes contains arbitrarily long arithmetic progressions.
\end{theorem}

\begin{proof}[Proof of Theorem~\ref{bounded}]
Let $r=\mbox{Card}(A).$
By hypothesis there exists a positive integer $M$ such that
$\abp_{\omega}(n)\leq M$ for every $n\geq 1.$ Thus for every
$(a_1,a_2,\ldots ,a_r),(b_1,b_2,\ldots ,b_r)\in \Psi_{\omega}(n),$
we have $|a_i-b_i|\leq M$ for $1\leq i\leq r.$ Let $\alpha _1,\alpha_2,\ldots ,\alpha_r$ and $N$ be as in Lemma~\ref{congo}.
Up to word isomorphism, we may regard $\omega=\omega_1\omega_2\omega_3\ldots $ as an infinite word on the alphabet $\{\alpha_1,\alpha_2,\ldots ,\alpha_r\}.$
For each $1\leq i\leq j$ we set $\omega_{[i,j]}=\omega_i\ldots \omega_j,$ and \[\sum \omega_{[i,j]}=\omega_i+\cdots +\omega_j.\]

Now consider the following function:
\[\nu :\{1,2,3,\ldots \}\rightarrow \{0,1,\ldots, N-1\}\]
defined by
\[\nu(t)=\sum\omega_{[1,t]} \pmod{N}.\]
By van der Waerden's Theorem, for every positive integer $k,$ there exist
positive integers $t_0$ and $s$ such that
\begin{equation}\label{zanzi}\nu(t_0)=\nu(t_0+s)=\nu(t_0+2s)=\cdots =\nu(t_0+ks).\end{equation}

\noindent For each $1\leq j\leq k,$ set \[\omega^{[j]}=\omega_{[t_0+(j-1)s+1,t_0+js]}.\]
Then by \eqref{zanzi}
\begin{equation}\label{sum}\sum \omega^{[j]}\equiv 0 \pmod {N}\end{equation}
for each $1\leq j\leq k.$

We will show that
$\Psi(\omega^{[j]})=\Psi(\omega^{[1]})$ for every $1\leq j\leq k.$
Set $\Psi(\omega^{[j]})=(a_1^{[j]},a_2^{[j]},\ldots ,a_r^{[j]}).$ By \eqref{sum}
\[\sum _{i=1}^ra_i^{[j]}\alpha_i\equiv \sum_{i=1}^ra_i^{[1]}\alpha_i \pmod{N},\] and hence
\[\sum _{i=1}^r (a_i^{[j]}-a_i^{[1]})\alpha_i \equiv 0 \pmod {N}.\]
Moreover as $|\omega^{[j]}|=|\omega^{[1]}|$ for each $1\leq j\leq k,$ we have that $|a_i^{[j]}-a_i^{[1]}|\leq M.$ By Lemma~\ref{congo} we deduce that
$a_i^{[j]}-a_i^{[1]}=0$ and hence that $\Psi(\omega^{[j]})=\Psi(\omega^{[1]})$ for every $1\leq j\leq k.$ Thus  the factor $\omega^{[1]}\omega^{[2]}\cdots \omega^{[k]}$ is an Abelian $k$- power of $\omega .$

\end{proof}

\begin{corollary} Let $k$ be a positive integer and $\omega$ an infinite word which avoids Abelian $k$-powers. Then $\omega$ is arbitrarily imbalanced, that is, $\omega$ is not $C$ balanced for any positive integer $C.$
\end{corollary}

\begin{proof} This follows immediately from Theorem~\ref{bounded} together with Lemma~\ref{bounded=balanced}.
\end{proof}

\noindent Theorem~\ref{bounded} naturally gives rise to the following  question.

\begin{problem}\label{mainprob}
Does every uniformly recurrent infinite word with bounded Abelian complexity begin with an Abelian $k$-power for each positive integer $k$? 
\end{problem}
This problem seems difficult to solve even for some very well-known words.
For example, it can be shown that every shift of the Tribonacci word begins in an Abelian $k$-power for all $k$,
but we do not know  if this is true for every word in its orbit closure.
It can also be shown that every shift of the Thue-Morse word begins in an Abelian $6$-power,
but again, we do not know whether this holds for larger Abelian powers or for every word in the orbit closure of the Thue--Morse word.
In the next section, however,  we show that the question does hold true in the case of Sturmian words, even in a very strong form.

\section{Abelian repetitions in Sturmian words}\label{SturmianSec}

In this section, we prove the following theorem, thereby answering Open Problem~\ref{mainprob} in the affirmative in the case of Sturmian words.

\begin{theorem} \label{T:ab-k pow in Sturm}
For every Sturmian word $\omega$ and every positive integer $k$, there exist two integers $\ell_1$ and $\ell_2$ such that 
each position in $\omega$ begins in an Abelian $k$-power $U_1U_2\ldots U_k$ with Abelian period $\ell_1$ or $\ell_2,$ that is $|U_i|\in \{\ell_1,\ell_2\}.$
In particular, every Sturmian word begins in an Abelian $k$-power for all positive integers $k$.
%
\end{theorem}


\begin{remark}
Theorem~\ref{T:ab-k pow in Sturm} should be contrasted with the fact that there exist Sturmian words whose initial critical exponent is equal to $2$ (see \cite{BerHolZam}). We also note that the existence of {\it two} Abelian periods in Theorem ~\ref{T:ab-k pow in Sturm} is optimal in the sense that any word with the property that every position starts with an Abelian $k$-power with a fixed Abelian period $m$ is necessarily ultimately periodic.
\end{remark}

Let us recall \cite{BerSee2002} that an infinite word is Sturmian if and only if its set of factors coincides with that of a characteristic word. A \emph{characteristic word} is an infinite word $c_{\alpha}$, depending on an irrational number $\alpha$ with $0<\alpha <1$, such that
\[
c_{\alpha}(n) = \lfloor \alpha(n+1) \rfloor - \lfloor \alpha n \rfloor
\]
for all $ n\geq 1$. Equivalently  the characteristic word $c_{\alpha}$ can also  be defined by
\[
c_{\alpha}(n) = R^n(\alpha) \qquad (n\geq 1),
\] 
where 
\[
R^n(\alpha) = 
\left\{
\begin{array}{rl}
0 & \mbox{if $\{n\alpha\} < 1 - \alpha$;}\\
1 & \mbox{otherwise.}
\end{array}
\right.
\]

\noindent For positive integers $i\leq j$ we put 
\[c_{\alpha} [i,j]=c_{\alpha}(i)c_\alpha(i+1)\cdots c_{\alpha}(j).\]

In the proof of the next lemma, we use the well-known fact that, for all real numbers $x,y,\alpha$, we have 
\begin{equation}\label{absfrac}
\abs{\fpart{x + \alpha}  - \fpart{y + \alpha} }=
\begin{cases}
\abs{\fpart{x} - \fpart{y}} & \text{or} \\ 
1-\abs{\fpart{x} - \fpart{y}}. &\\
\end{cases}
\end{equation}

\begin{lemma}\label{Rlemma}
Suppose that  $\alpha$ is irrational with $0<\alpha<1/2$.
Suppose also that $i,j\geq 1$ are integers such that $\abs{ \{i \alpha \} - \{j\alpha\} }< \alpha$.
If $k\geq 1$ is  an integer such that 
\[    
R^{i+k}(\alpha) = R^{j+k}(\alpha),
\]
then the words
\begin{equation}\label{abelLemma}
c_{\alpha}[i,i+k]\quad \mbox{and} \quad 
c_{\alpha}[j,j+k]
\end{equation}
are Abelian equivalent.

\end{lemma}
\begin{proof}
In Eq.~\eqref{abelLemma}, denote the word on the left by $W_i$ and the one on the right by $W_j$. 
Since $R^{i+k}(\alpha) = R^{j+k}(\alpha)$,  we have 
\[
\abs{ \{(i+k)\alpha\} - \{(j+k)\alpha\}  } < \max\{ \alpha, 1-\alpha\} = 1 - \alpha.
\]
Hence from $\abs{ \{i \alpha \} - \{j\alpha\} }< \alpha$ and Eq.~\eqref{absfrac}, we deduce that 
\[
\abs{ \{i \alpha \} - \{j\alpha\} } = \abs{ \{(i+k)\alpha\} - \{(j+k)\alpha\}  }.
\]
Hypothesis $R^{i+k}(\alpha) = R^{j+k}(\alpha)$ also implies that\\
\[ \abs{ \{(i+k)\alpha\} - \{(j+k)\alpha\}  }  = \abs{ \{(i+k+1)\alpha\} - \{(j+k+1)\alpha\}  }\] and so 
\[
\abs{ \{(i+k+1)\alpha\} - \{(j+k+1)\alpha\}  } < \alpha.
\]
The number of letters 1 in $W_i$ is given by
\[
\abs{W_i}_1 = \sum_{h = 0}^k c_{\alpha}(i+h) = \sum_{h=0}^k \bigl( \lfloor(i+h+1)\alpha\rfloor  - \lfloor (i + h)\alpha\rfloor \bigr) =  \lfloor(i+k+1)\alpha\rfloor  - \lfloor i\alpha\rfloor,
\]
and similarly for $W_j$.
These give
\begin{align*}
\bigl\lvert \abs{W_j}_1 - \abs{W_i}_1 \bigr\rvert& =
\bigl\lvert  \bigl( \lfloor(j+k+1)\alpha\rfloor  - \lfloor j\alpha\rfloor\bigr) - \bigl( \lfloor  (i+k+1)\alpha\rfloor -  \lfloor i\alpha\rfloor  \bigr) \bigr\rvert \\
   &= \bigl\lvert  \lfloor i\alpha\rfloor - \lfloor j \alpha\rfloor + \lfloor(j+k+1)\alpha\rfloor - \lfloor(i+k+1)\alpha\rfloor \bigr\rvert\\
   &= \bigl\lvert i\alpha -\{i\alpha\} - j\alpha + \{ j\alpha \} + (j+k+1)\alpha - \\
   & \quad\quad  \{(j+k+1)\alpha\} - (i+k+1)\alpha + \{(i+k+1)\alpha\} \bigr\rvert \\
   &\leq \bigl\lvert  \{j\alpha\} -\{ i\alpha\} \bigr\rvert +  \bigl\lvert  \{(i+k+1)\alpha\} -\{(j+k+1)\alpha\} \bigr\rvert \\
   & < 2 \alpha <1.
 \end{align*}
Consequently, $W_i$ and $W_j$ are Abelian equivalent.
\end{proof}

%

In the proof  of the next theorem, we will use the following fact:
for all real numbers $x,y$ and integers $p\geq 1$, we have 
\begin{equation}\label{fpart}
\fpart{x + py} =  \fpart{x} + p\fpart{y} - (p-q),
\end{equation}
where $q$ is the integer for which 
\begin{equation*}
p-q \leq \fpart{x} + p\fpart{y} < p - q + 1.
\end{equation*}

\begin{theorem}\label{abelthm}
Let $\alpha$ be irrational with $0<\alpha < 1$. For all positive integers $k$ and $i$,
there is an Abelian $k$-power occurring at position~$i$ in $c_{\alpha}$.
\end{theorem}

\begin{proof}
It is a well-known fact that $c_{1-\alpha}$ is obtained from $c_{\alpha}$ by exchanging letters 0 and 1  (see \cite[Cor. 2.2.20]{BerSee2002}). Therefore,
without loss of generality, we may suppose that $\alpha<1/2$. 
To show that there exists an Abelian $k$-power at the $i$th position in $c_{\alpha}$, let us choose a real number $\delta$ with
\begin{equation}\label{interval1}
 0 < \delta < \alpha.
 \end{equation}
 We have two cases to consider:

\textbf{Case 1}. We have
\begin{equation}\label{range}
 0 \leq \{ i \alpha \} < \alpha -\delta, \quad  \mbox{or} \quad  \alpha \leq \{ i\alpha \} < 1-\delta.
\end{equation}
None of these intervals is empty by Eq.~\eqref{interval1} and since $\alpha<1/2$. 








Let $\ell$  be a positive integer such that 
\begin{equation}\label{ellequ1}
\{ \ell \alpha \} < \delta/k.
\end{equation}
This choice of $\ell$ and the inequality $\{i\alpha\} < 1 - \delta$ imply that, for all $0 \leq j \leq k$, we have 
\[
\{i \alpha\} + j \{\ell \alpha\} < 1 - \frac{(k-j)\delta}{k} \leq 1,
\]
so from Eq.~\eqref{fpart} we get 
\[
\bigl\{ \bigl(i+j\ell\bigr)\alpha \bigr\}  =  \{ i \alpha \}  + j\{\ell \alpha  \}.
\]
This and Eq.~\eqref{ellequ1} give
\begin{equation}\label{incrineq}
\{i\alpha\} \leq  \{(i+j\ell) \alpha \}   < \{i \alpha \} + \delta
\end{equation}
for all $0\leq j \leq k$.
These imply that if $\{i \alpha \} < \alpha - \delta$, then we have
\[
R^{i+\ell -1}(\alpha) = R^{i+2\ell - 1}(\alpha) = \cdots= R^{i+k\ell - 1}(\alpha) = 1.
\]
Similarily, if $\{i \alpha\} \geq \alpha$, inequalities~\eqref{incrineq} imply
\[
R^{i+\ell -1}(\alpha) = R^{i+2\ell - 1}(\alpha) = \cdots= R^{i+k\ell - 1}(\alpha) = 0.
\]
By these observations, Lemma~\ref{Rlemma} says that the  $k$ words 
\[
c_{\alpha}[i + j\ell, i + (j+1)\ell -1] \qquad (0 \leq j \leq k-1)
\]
%
of length $\ell$ are Abelian equivalent. Therefore  $c_{\alpha}[\, i, i+ k\ell -1 \,]$ is an Abelian $k$-power.

\textbf{Case 2}. We have
\begin{equation}\label{range2}
 \alpha - \delta \leq \{ i \alpha \} < \alpha, \quad  \quad \mbox{or} \quad  1-\delta \leq \{ i\alpha \} < 1.
\end{equation} 










Now we  choose a positive integer $\ell$ such that 
\begin{equation}\label{ellequ2}
1-  \frac{  \alpha-\delta}{k}         < \{\ell \alpha \} < 1.
\end{equation}
Since  $\{i\alpha\} \geq \alpha -\delta$, we have, for all $0 \leq j \leq k$, 
\[
1 + j > \{i \alpha\} + j \{\ell \alpha\} > j + \frac{(k-j)(\alpha-\delta)}{k} \geq j.
\]
Hence by Eq.~\eqref{fpart},
\[
\bigl\{ \bigl(i+j\ell\bigr)\alpha \bigr\}  =  \{  i \alpha \}  - j(1- \{\ell \alpha  \}),
\]
and so the lower bound for $\{\ell\alpha\}$ in Eq.~\eqref{ellequ2} gives
\begin{equation}\label{decrineq}
\{i\alpha\} - (\alpha -\delta) < \{(i+j\ell)\alpha \} \leq \{i\alpha \}
\end{equation}
for all $0\leq  j \leq k$.

%

Now if $\{i \alpha \} \geq 1 - \delta$,  we have
\[
\{ (i+j\ell)\alpha \} > 1 - \alpha > \alpha
\]
for all $0\leq  j \leq k$, and this gives
\[
R^{i+\ell -1}(\alpha) = R^{i+2\ell - 1}(\alpha) = \cdots= R^{i+k\ell - 1}(\alpha) = 0.
\]
Otherwise $\{i \alpha \} <  \alpha$,  and  we have
\[
R^{i+\ell -1}(\alpha) = R^{i+2\ell - 1}(\alpha) = \cdots= R^{i+k\ell - 1}(\alpha) = 1.
\]
By these equations and Eq.~\eqref{decrineq},  Lemma~\ref{Rlemma} implies that
the  $k$ words  
\[
c_{\alpha}[i + j\ell, i + (j+1)\ell - 1] \qquad (0 \leq j \leq k-1)
\]
%
of length $\ell$ are Abelian equivalent, and so the word $c_{\alpha}[\,i, i +k\ell -1\,]$ is an Abelian  $k$-power. 

\end{proof}

Let us denote by  $\alpha = [\, 0;a_1,a_2, a_3, \ldots \,]$
the continued fraction expansion of $\alpha$.  For $n\geq 0$, denote
\[
\frac{p_n}{q_n} = [\, 0;a_1, \ldots, a_n\, ],
\]
where ${\rm gcd}(q_n , p_n) = 1$.  
We will need  the next two basic  properties of  continued fractions. 
\begin{itemize}
\item[(1)] We have $q_{n+1} > q_n$ for all $n\geq 1$. 
\item[(2)] If $n\geq 0$ is even, we have
\begin{equation} \label{evencase}
0 <  \alpha - \frac{p_n}{q_n}  < \frac{1}{q_n q_{n+1}};
\end{equation}
and if $n$ is odd, we have
\begin{equation}\label{oddcase}
0 <  \frac{p_n}{q_n} - \alpha  < \frac{1}{q_n q_{n+1}}.
\end{equation}
\end{itemize}

\begin{proof}[Proof of Theorem~\ref{T:ab-k pow in Sturm}]
Let $\alpha$ denote the slope of  the Sturmian word $\omega$. The set of factors of~$\omega$ coincides with that of the characteristic word $c_{\alpha}$, so it suffices to prove the claim for~$c_{\alpha}$.

As in the proof of Theorem~\ref{abelthm}, we assume that $\alpha < 1/2$ and let $\delta$ be  a real number with  $0 < \delta < \alpha$. 
As $\lim_{n\rightarrow \infty}q_n=+\infty$ there exists an even  integer $n\geq 0$ such that 
\[
q_{n+1} >   \frac{k}{ \min\{ \delta,  \alpha - \delta\} }.
\]
Eq.~\eqref{evencase} then implies that 
\[
\{ q_n \alpha\} = q_n \alpha - p_n < \frac{1}{q_{n+1}} < \frac{ \min\{ \delta, \alpha - \delta\} }{k} \leq \frac{\delta}{k}.
\]
Consequently, if $\{i\alpha\}$ is contained in one of the intervals of Case 1 in the proof of Theorem~\ref{abelthm}, we may choose  
$\ell = q_n$, and then  the word $c_{\alpha}[\, i, i+ k\ell -1 \,]$ is an Abelian $k$-power with Abelian period~$q_n$.

If $\{i \alpha\}$ is in one of the intervals of Case 2,  we may choose $\ell = q_{n+1}$. Indeed, then by Eq.~\eqref{oddcase}, we have
\[
1 - \{ q_{n+1}\alpha\} =  p_{n+1} - \alpha q_{n+1} < \frac{1}{q_{n+2}} < \frac{1}{q_{n+1}} <  \frac{ \min\{ \delta, \alpha - \delta\} }{k}.
\]
That is
\[
1 - \frac{\alpha - \delta}{k}   < \{ q_{n+1}\alpha \} < 1.
\]
Now by the proof of Case 2, the word $c_{\alpha}[\, i, i+ k\ell -1 \,]$ is an Abelian $k$-power with Abelian period $q_{n+1}$; 
the claim follows.
\end{proof}

\begin{remark}
The property of Sturmian words given in Theorem~\ref{T:ab-k pow in Sturm} does  not provide a characterization of  Sturmian words. For instance, this same property is also satisfied by any word of the form $f(\omega)$ with $\omega$ a Sturmian word and where $f$ is the morphism defined by $f(0) = 00$ and $f(1) = 01.$
\end{remark}

\begin{remark}
According to the terminology introduced in~\cite{Saari2009EJC}, Sturmian words are \emph{everywhere Abelian repetitive}.
\end{remark}



%


\bibliographystyle{plain}
\bibliography{RSZ3}

\end{document}